\documentclass{amsart}
\usepackage{amsfonts,amssymb,amscd,amsmath,enumerate,verbatim,calc}
\usepackage{xy}

\newcommand{\CM}{Cohen-Macaulay}

\newcommand{\n}{\mathfrak{n} }
\newcommand{\m}{\mathfrak{m} }
\newcommand{\M}{\mathfrak{M} }

\newcommand{\q}{\mathfrak{q} }

\newcommand{\Z}{\mathbb{Z} }

\newcommand{\rt}{\rightarrow}

\newcommand{\ov}{\overline}

\newcommand{\image}{\operatorname{image}}

\newcommand{\depth}{\operatorname{depth}}
\newcommand{\type}{\operatorname{type}}
\newcommand{\Ass}{\operatorname{Ass}}

\newcommand{\Ext}{\operatorname{Ext}}

\theoremstyle{plain}

\newtheorem{thm}{Theorem}[section]

\newtheorem{theorem}{Theorem}[section]

\newtheorem{lemma}[theorem]{Lemma}
\newtheorem{proposition}[theorem]{Proposition}

\theoremstyle{definition}

\newtheorem{s}[theorem]{}
\newtheorem{remark}[theorem]{Remark}
\newtheorem{example}[theorem]{Example}

\theoremstyle{remark}

\begin{document}
\title[$e_2 = e_1 - e  + 1$]{Cohen-Macaulay local rings with $e_2 = e_1 - e + 1$}
\author{Ankit Mishra}
\email{ankitmishra@math.iitb.ac.in}

\author{ Tony~J.~Puthenpurakal}
\email{tputhen@math.iitb.ac.in}

\address{Department of Mathematics, IIT Bombay, Powai, Mumbai 400 076}

\date{\today}


\begin{abstract}
In this paper we study \CM \ local rings  of dimension $d$, multiplicity $e$ and second Hilbert coefficient $e_2$ in the case $e_2 = e_1  - e  + 1$. Let $h = \mu(\mathfrak{m}) - d$. If $e_2 \neq 0$ then in our case we can prove that $\type A \geq e - h -1$. If $\type A = e - h -1$ then we show that the associated graded ring $G(A)$ is \CM. In the next case when $\type A = e - h$  we determine all possible Hilbert series of $A$. In this case we show that the Hilbert Series of $A$  completely determines $\depth G(A)$.
\end{abstract}

\maketitle	
\section{Introduction}

Let $(A,\mathfrak{m})$ be a Cohen-Macaulay local ring with maximal ideal $\mathfrak{m}$ and dimension $d$. Let $k=A/\mathfrak{m}$ be its residue field.  Let $G(A) = \bigoplus_{i\geq0}{\mathfrak{m}^i/\mathfrak{m}^{i+1}}$ be the associated graded ring of $A$ with respect to maximal ideal $\mathfrak{m}$. If $M$ is an $A$-module then let $\ell(M)$ denotes its length and $\mu(M)$ the number of its minimal generators. The Hilbert series of $A$ is $$H_A(z)=\sum_{i\geq0}\ell({\mathfrak{m}^i}/{\mathfrak{m}^{i+1}})z^i=\frac{h_A(z)}{(1-z)^d}$$
Where $h_A(z)=h_0+h_1z+\ldots+h_sz^s\in \mathbb{Z}[z]$ is known as $h$-polynomial of $A$. Here $h_A(1)\neq0$ , $h_0=1$ and $h=$ embedding codimension of $ A $= $ \mu(\mathfrak{m})-d$,

We define another numerical function $H^{(1)}(A,i)=\ell(A/{\mathfrak{m}^{i+1}})$, this function is known as Hilbert-Samuel function of $A$ with respect to $\mathfrak{m}$. The generating function of this function
$$H_A^{(1)}(z)=\sum_{i\geq0}H^{(1)}(A,i)z^i=H_A(z)/{(1-z)}$$
   There exists a polynomial $P_A(z)$ of degree $d$ such that $P_A(i)=H^{(1)}(A,i)$ for all $i\gg0$ and we know that $$P_A(z)=\sum_{i=0}^d (-1)^ie_i(A)\binom{z+d-i}{d-i}$$ The integers $e_i(A)$ are called the $i$-th Hilbert coefficient of $A$. If $g$ is a polynomial we use $g^{(i)}$ to denote $i$-th formal derivative of $g.$ It is easy to see that $e_i=h_A^{(i)}(1)/{i!}$ for  $0 \leq i\leq d.$ It is also convenient to set $e_i=h_A^{(i)}(1)/{i!}$ for all $i\geq 0$. Here $e_0$ is known as multiplicity of $A$.

In \cite{It} Itoh proved that $e_2\geq e_1-e+1$. In \cite[6.20]{V} it was conjectured that if $e_2=e_1-e+1$ then $s\leq2$ and $G(A)$ is Cohen-Macaulay. This conjecture is true for $d=0,1$ but for $d=2$ there is a counter-example given by Wang in \cite[3.10]{CPR}.

We define type of $A$ as length of $\Ext^d_A(k,A)$ where $k=A/\mathfrak{m}.$

If $A$ is Gorenstein ring with $e_2=e_1-e+1$ then $G(A)$ is Gorenstein. Infact, in this case $e=h+2$ so from \cite[Theorem 3.4]{S}, $G(A)$ is Gorenstein ring. Gorenstein rings are Cohen-Macaulay of type one. So it motivated us to study the   behaviour of $G(A)$ by applying some condition on the type of $A$.

If $(A,\mathfrak{m})$ is Cohen-Macaulay local ring  with $e_2=e_1-e+1 \neq0$, then it is easy to prove type$(A)\geq e-h-1$, see \ref{greater}.

We will first consider the case when type$(A)=e-h-1$ and prove,
\begin{thm}

	Let $(A,\mathfrak{m},k)$ be a Cohen-Macaulay local ring with dimension $d$ and suppose $e_2=e_1-e+1\neq0$.
	If type$(A)=e-h-1$, then $G(A)$ is Cohen-Macaulay.
\end{thm}

Next we will consider the case when type$(A)= e-h$, in this case we will prove that if $\depth G(A) = 0$ then $\widetilde{G}(A)$ is Cohen-Macaulay, where \\ $\widetilde{G}(A)=\bigoplus_{n \geq 0}{\widetilde{\mathfrak{m}^n}}/{\widetilde{\mathfrak{m}^{n+1}}}$ is associated graded ring of $A$ with respect to Ratliff-Rush filtration $\{\widetilde{\mathfrak{m}^n}\}_{n\geq0}$. We prove

\begin{thm}
	Let $(A,\mathfrak{m})$ be Cohen-Macaulay local ring of dimension $d\geq2$ with type$(A)=e-h$ and $e_2=e_1-e+1\neq0$. Suppose $x_1,\ldots,x_d$ is $A$-superficial sequence. Set $J=(x_1,\ldots,x_d)$. If depth $G(A)=0$ then
	\begin{enumerate}[\rm (1)]
		\item $\widetilde{\mathfrak{m}^2} \neq \mathfrak{m}^2$ and  $\ell({\widetilde{\mathfrak{m}^2}}/{\mathfrak{m}^2})=1$.
		\item $\widetilde{G}(A)$ is Cohen-Macaulay.
		\item  $\widetilde{\mathfrak{m}^{n+1}}=J\widetilde{\mathfrak{m}^n}$ for all $n\geq2$.
		\item $\widetilde{\mathfrak{m}^j}=\mathfrak{m}^j$ for all $j\geq3$ and $\ell({\mathfrak{m}^3/{J\mathfrak{m}^2}})=d$.
        \item $G(A)$ is generalized \CM.
	\end{enumerate}
\end{thm}
By using above theorem we will give $h$-polynomials in this case.
	
\begin{thm}
	Let $(A,\mathfrak{m})$ be Cohen-Macaulay local ring of dimension $d\geq2$ with type$(A)=e-h$ and $e_2=e_1-e+1\neq0$. Now if depth $G(A)=0$ then $$h_A(z)=1+(h-1)z+(e-h)z^2+z(1-z)^{d+1}.$$
	In general, if depth $G(A)=n\neq d$ then $n\leq d-2$ and $h_A(z)=1+(h-1)z+(e-h)z^2+z(1-z)^{d-n+1}$, and  if $G(A)$ is Cohen-Macaulay then $h_A(z)=1+hz+(e-h-1)z^2$.
\end{thm}
The result given below gives (with above set-up) a bound on the dimension, above which depth $G(A)$ is always positive.
\begin{thm}
	Let $(A,\mathfrak{m})$ be Cohen-Macaulay local ring with dimension $d$. Suppose $e_2=e_1-e+1\neq0$, type$(A)=e-h=k$. Now if $d>(k-1)^{<2>}$ then depth of $G(A)$ is non-zero.
\end{thm}
	We now describe in brief the contents of this paper. In section 2 we describe some preliminary results that we need. In section 3 we will discuss the case type$(A)=e-h-1$ and give the proof of Theorem 1.1. In section 4 we will discuss the case type$(A)=e-h$ and give the proofs of Theorems 1.2, 1.3 and 1.4. Finally in last section we will give some examples.

\section{Preliminaries}

In this article we will consider $(A,\mathfrak{m})$ as Noetherian  local ring of dimension $d$ with residue field $k=A/\mathfrak{m}$.
	 \s
	 	If $a$ is a nonzero element of $A$ then there is a largest integer $n$ such that $a\in \mathfrak{m}^n$. We will denote the image of $a$ in $\mathfrak{m}^n/\mathfrak{m}^{n+1}$ as $a^*$.

\s
If $d\geq1$	an element $x\in \mathfrak{m}$ is said to be $A$-superficial if there exists $c$ and $n_0$ such that
		$(\mathfrak{m}^{n+1} : x)\cap \mathfrak{m}^c =\mathfrak{m}^n$ for all $n\geq n_0$.
		
		 Superficial elements always exists when residue field is infinite.
		 If depth $A>0$ then every $A$-superficial element is also $A$-regular. If $x$ is $A$-superficial and regular then $(\mathfrak{m}^{n+1} : x)=\mathfrak{m}^n$ for $n\gg0.$ \\
	A sequence $x_1,\ldots,x_r$ in a local ring $(A,\mathfrak{m})$ is said to be an $A$-superficial sequence if $x_1$ is $A$-superficial and $x_i$ is $A/{(x_1,\ldots,x_{i-1})}$-superficial for $2\leq i\leq r.$
	
\begin{remark}
	If the residue field of $A$ is finite then we can  replace $A$ by $A'=A[X]_S$ where $S=A[X]\backslash \mathfrak{m}A[X]$. The residue field of $A'$ is $k(X)$, the field of rational functions over $k.$ Furthermore, $$H(A',n)= H(A,n) \quad \text{for all}\quad n\geq0\quad\text{and depth}\ G(A')= \text{depth}\ G(A).$$
	\end{remark}

\s Let $(A,\mathfrak{m})$ be a Noetherian local ring, and $I$ a proper ideal of $A$, then  $y_1,\ldots,y_n$ are said to be analytically independent in $I$ if whenever $F(Y_1,\ldots,Y_n)\in A[Y_1,\ldots,Y_n]$ is a homogeneous polynomial of degree $r$ such that $F(y_1,\ldots,y_n)\in I^r\mathfrak{m}$, then all the coefficients of $F$ are in $\mathfrak{m}$.\\
	Let $(A,\mathfrak{m})$ be a Cohen-Macaulay local ring of dimension $d$, and $x_1,\ldots ,x_d$ be an $A$-superficial sequence then $J=(x_1,\ldots ,x_d)$ is a minimal reduction of $\mathfrak{m}$ (for instance see \cite[Lemma 6]{Pu0}), and $x_1,\ldots ,x_d$ are analytically independent in $\mathfrak{m}$;  see \cite[Lemma 2]{NR}.

\s \label{superf}
  We need to recall the construction of a superficial element as we need  it in the proof of Theorem \ref{B}.
	Assume the residue field of $A$ is infinite. Set $G=G(A)$ and let $\M$ be its maximal homogeneous ideal. We denote the $k$-vector space $\mathfrak{m}/\mathfrak{m}^2$  by $V$. If $P$ is a prime ideal of $G$ with $P\neq\M$ then $P\cap V\neq V$. As $k$ is infinite we get $$V\neq \bigcup_{P\in \Ass G, P\neq \M}P\cap V$$
	If $x\in \mathfrak{m}$ is such that $$x^*\in V \backslash \bigcup_{P\in \Ass G, P\neq \M}P\cap V$$ then $x$ is $A$-superficial.

\s
	Let $A$ be Cohen-Macaulay with  $d=1$. Let $x$ be an $A$-superficial element. Set $\rho_i = \ell(\mathfrak{m}^{i+1}/{x\mathfrak{m}^i})$ for all $i\geq0$ then
\begin{enumerate}

	\item If deg $h_A(z) =s$ then $e=\rho_0+1, \rho_i=0$ for all $i\geq s$, $e_1=\sum_{i=0}^{s-1}\rho_i$ and $e_2=\sum_{i=1}^{s-1}i\rho_i$.
	\item If $e_2=e_1-e+1$ then $s\leq2$ and $G(A)$ is Cohen-Macaulay, see \cite[Proposition 6.21]{V}.
\end{enumerate}

\s
	If depth $A>0$ then for $i\geq1$ we define Ratliff-Rush closure of $\mathfrak{m}^i$ by $$\widetilde{\mathfrak{m}^i}=\bigcup_{k\geq1}{(\mathfrak{m}^{i+k}:\mathfrak{m}^k)}.$$ Then $\{\widetilde{\mathfrak{m}^i}\}_{i\geq0}$ is called the Ratliff-Rush filtration of $A$ with respect to $\mathfrak{m}$.\\ The associated graded ring of $A$ with respect to Ratliff-Rush filtration $\{\widetilde{\mathfrak{m}^i}\}_{i\geq0}$  is  $\displaystyle{ \widetilde{G}(A)} = \bigoplus_{n \geq 0}{\widetilde{\mathfrak{m}^n}}/{\widetilde{\mathfrak{m}^{n+1}}} $.\\ We have
	\begin{enumerate}
	  \item
	$\widetilde{\m^i}\widetilde{\m^j} \subseteq \widetilde{\m^{i+j}}$ for all $ i, j \geq 0$.
	\item
	$\widetilde{\m^i} = \m^i$ for all $i \gg 0$.
	\item
	If $x$ is $A$-superficial then $(\widetilde{\m^{i+1}} \colon x) = \widetilde{\m^i}$ for all $i \geq 0$.
	\item
	$\depth G(A) > 0$ if and only if $\widetilde{\m^i} = \m^i$ for all $i \geq 1$.
	\end{enumerate}

\begin{lemma}\label{sigma}\cite[Theorem 3]{It}
	Let $(A,\mathfrak{m})$ be Cohen-Macaulay local ring with dimension two, and $x,y$ be $A$-superficial sequence. Set $J=(x,y)$ and $\sigma_n=\ell(\widetilde{\mathfrak{m}^{n+1}}/{J\mathfrak{m}^n})$. Then $$e_1=\sum_{n\geq0}{\sigma_n} \quad \text{and} \quad e_2=\sum_{n\geq 1}{n\sigma_n} $$
\end{lemma}
 Let $(A,\m)$ be a \CM \ local ring of dimension $d$. Let $x_1,\ldots, x_d$ be an $A$-superficial sequence. Set $J = (x_1,\ldots, x_d)$. Then $\widetilde{\mathfrak{m}^2}\cap J = J\mathfrak{m}$, because if $a=r_1x_1+\ldots +r_dx_d \in \widetilde{\mathfrak{m}^2}$, then there exists an $n\geq0$ such that $ax_1^n\in \mathfrak{m}^{n+2}$. Since $x_i$'s are analytically independent  we have $r_i\in \mathfrak{m}$ for $i=1,\ldots,d$.

\s \label{mod-sup}Let $x \in \m$ be $A$-superficial and $A$-regular. Set $B = A/(x)$. Also set $b_n = \ell((\m^{n+1} \colon x)/\m^n)$ and
$b_x(z) = \sum_{n \geq 0}b_n z^n$, then $b_x(z) \in \Z[z]$. Furthermore, we have
\[
h_A(z) = h_B(z) - (1-z)^db_x(z).
\]
So $e_i(A) = e_i(B)$ for $0 \leq i \leq d - 1$. Also $x^*$ is $G(A)$-regular if and only if $b_x(z) = 0$. In this case we have $G(B) = G(A)/(x^*)$.

\s	 Assume $A$ is Cohen-Macaulay local ring of dimension $d$ and     $x_1,\ldots,x_r$ is an $A$-superficial sequence with $r\leq d$. Set $B=A/{(x_1,\ldots,x_r)}$. Then
	\begin{enumerate}
			\item $e_i(A)=e_i(B)$ for all $i=0,\ldots,d-r$.
		\item depth $G(A)\geq r$ if and only if $x_1^*,\ldots,x_r^*$ is $G(A)$-regular. In this case we have $G(B)=G(A)/{(x_1^*,\ldots,x_r^*)}$.
		\item (Sally descent:) depth $G(A)\geq r+1$ if and only if depth $G(B)\geq1$.
	
	\end{enumerate}

\s \label{quot}
	 Assume $(A,\m)$ is \CM \ of dimension $d \geq 2$. Let $x_1,\ldots, x_d$ be an $A$-superficial sequence. Set $J = (x_1,\ldots, x_d)$ and $(B,\n) = (A/(x_1), \m/(x_1))$.  If $I$ is an ideal in $A$ then set $\ov{I}$ to be its image in $B$. Notice $\ov{\widetilde{\m^i}} \subseteq \widetilde{\n^i}$. So we have a natural map
	$$\eta_i \colon \widetilde{\m^{i+1}}/ J\widetilde{\m^i}  \rt \widetilde{\n^{i+1}}/\ov{J} \widetilde{\n^i} \quad \text{for all} \ i \geq 0. $$

	We show
	\begin{proposition}\label{nn}
		(with setup as in \ref{quot})
		\begin{enumerate}[\rm (1)]
			\item
			If $\ov{\widetilde{\m^s}} = \widetilde{\n^s}$ for some $s$ then $\eta_s$ is injective.
			\item
			If $\ov{\widetilde{\m^j}} = \widetilde{\n^j}$ for $j = s, s+1$ then $\eta_s$ is bijective.
		\end{enumerate}
	\end{proposition}
	\begin{proof}
		(1) Suppose $\eta_s( p + J \widetilde{\m^s}) = 0$. Then $\ov{p} = \sum_{ t = 2}^{d}a_tx_t$ for some $a_t \in \widetilde{\n^s}$. As $\ov{\widetilde{\m^s}} = \widetilde{\n^s}$ there exists $b_t \in \widetilde{\m^s}$ with $\ov{b_t} = a_t$ for $t = 2,\ldots, d$. So
		\[
		p = \sum_{t =2}^{d} b_tx_t +   rx_1 \quad \text{for some} \ r \in A.
		\]
		We get  $xr \in \widetilde{\m^{s+1}}$. So $r \in \widetilde{\m^s}$. Thus $p \in J \widetilde{\m^s}$. It follows that $\eta_s$ is injective.
		
		(2) As $\ov{\widetilde{\m^{s+1}}} = \widetilde{\n^{s+1}}$ then clearly $\eta_s$ is surjective. As $\ov{\widetilde{\m^s}} = \widetilde{\n^s}$ we get by (1) that $\eta_s$ is injective. So $\eta_s$ is bijective.
	\end{proof}

\begin{lemma}\label{dep02}
	Let $(A,\mathfrak{m})$ be Cohen-Macaulay local ring with dimension $d=2$, and if $e_2=e_1-e+1$ then depth $G(A) =0$ or $2$.
	
\end{lemma}
\begin{proof}
	Let $x$ be an $A$-superficial element and set $B=A/(x)$. Now if depth $G(A)\neq0$ then $e_2(A)=e_2(B),e_1(A)=e_1(B)$ and $e(A)=e(B)$. So $B$ is Cohen-Macaulay local ring of dimension one and $e_2(B)=e_1(B)-e(B)+1$.  So  $\depth G(B)=1$  see \cite[Prop 6.21]{V}. By Sally-descent,
 $\depth G(A)=2$.
\end{proof}

Now we will prove a result which we will use in the proof of Theorem \ref{B}.
\begin{lemma}\label{new}
	Let $(A,\mathfrak{m})$ be a Cohen-Macaulay local ring of dimension $d\geq2$ with an infinite residue field. Let $x_1,\ldots,x_d$ be an $A$-superficial sequence. Set $J=(x_1,\ldots,x_d)$, $(B,\mathfrak{n})=(A/(x_1),\mathfrak{m}/(x_1))$ and $\overline{J}=(x_2,\ldots,x_d)B$. Then we have an exact sequence  $$0\rightarrow \frac{(\mathfrak{m}^3:x)}{\mathfrak{m}^2}\xrightarrow{f}\frac{\mathfrak{m}^3}{J{\mathfrak{m}^2}}\xrightarrow{g}\frac{\mathfrak{n}^3}{\overline{J}{\mathfrak{n}^2}}\rightarrow0 $$ where $f(a+\mathfrak{m}^2)=ax_1+J\mathfrak{m}^2$ and $g$ is natural surjection.
	
\end{lemma}
\begin{proof}
 (for $d=2$ case one can also see \cite[Lemma 2.2]{rv}).\\
Surjectivity of $g$ is obvious. Furthermore clearly $g\circ f=0$. \\
Let $\alpha \in \ker g$. Then $\alpha= b+J\mathfrak{m}^2$ where $b\in \mathfrak{m}^3$, and there is a $c\in J\mathfrak{m}^2$ such that $b-c=rx_1$ with $r\in A$. So $r\in (\mathfrak{m}^3 : x_1)$.
Notice $f(r+\mathfrak{m}^2)= rx_1+J\mathfrak{m}^2=b+J\mathfrak{m}^2=\alpha$. This gives us $\ker g = \image f$. \\
 To prove exactness of above complex we  only have to show that $f$ is injective. Let $a\in (\mathfrak{m}^3:x_1)$ and $ax_1\in J\mathfrak{m}^2.$ So $ax_1=a_1x_1+\ldots+a_dx_d$ with $a_i\in \mathfrak{m}^2$ for $i=1,\ldots,d$. This implies $(a-a_1)\in (x_2,\ldots,x_d)$, because $x_2,\ldots,x_d,x_1$ is a regular sequence.
 Therefore
$ a-a_1= r_2x_2+\ldots+r_dx_d$ for some $r_i\in A$ for all $i=2,\ldots,d$. So we have
$ ax_1-a_1x_1=r_2x_2x_1+\ldots+r_dx_dx_1$. Therefore,
$$ r_2x_2x_1+\ldots+r_dx_dx_1\in \mathfrak{m}^3. $$
As $x_i$'s are analytically independent we get
$ r_i\in \mathfrak{m}$ for all $i=2,\ldots,d$.
Since
$ a=a_1+r_2x_2+\ldots +r_dx_d $, so
$ a \in \mathfrak{m}^2$. This gives injectivity of  $f$.
\end{proof}

\begin{proposition}\label{GCM}
	Set $G=G(A)$ and $\widetilde{G}=\widetilde{G}(A)$. Assume $\widetilde{G}$ is Cohen-Macaulay. Then
	\begin{enumerate}[\rm (1)]
		\item $G$ is generalized Cohen-Macaulay.
		\item dim $G/P=d$ for all minimal primes  $P$ of $G$.
	\end{enumerate}
\end{proposition}

\begin{proof}
	(1) We note that $\widetilde{G}$ is finitely generated as a $G$-module. Furthermore  the natural map $G \rt \widetilde{G}$ induces an exact sequence of $G$-modules
	\[
	0 \rt U \rt G \rt \widetilde{G} \rt V \rt 0,
	\]
	where $U, V$ are of finite length. Let $\M$ be the maximal homogeneous ideal of $G$. If $P$ is a prime ideal in $G$ with $P \neq \M$ then $G_P \cong \widetilde{G}_P$. It  follows that $G_P$ is \CM \ for all $P \neq \M$. So $G$ is generalized \CM.
	
	(2) Let $P$ be a minimal prime of  $G$. Then $P \neq \M$. Furthermore $G_P \cong \widetilde{G}_P$. So $P$ is an associate prime of
	$\widetilde{G}$. The result now follows from \cite[2.1.2(a)]{BH}.
\end{proof}	
\begin{proposition}\label{greater}
	Let $(A,\mathfrak{m},k)$ be a Cohen-Macaulay local ring of dimension $d$ and suppose $e_2=e_1-e+1\neq0$, then type$(A)\geq e-h-1.$
\end{proposition}
\begin{proof}
	We will use induction to prove this result.
	Now for $d=0$ we know that $s\leq2$ see\cite[Prop 6.21]{V}.  So $\mathfrak{m}^3=0$.  Therefore  $\mathfrak{m}^2 \subseteq $ Soc$(A)$. Furthermore notice  $\mathfrak{m}^2\neq0$ because if $\mathfrak{m}^2=0$ then $A$ has minimal multiplicity and this is not possible as $e_2\neq0$. So $\mathfrak{m}^2\neq0$ and  as $\ell(\mathfrak{m}^2)=e-h-1$ we get $\type A\geq e-h-1$.

	Now assume $d\geq1$. Let  $x_1,\ldots,x_d$ be an $A$-superficial sequence. Set $J=(x_1,\ldots,x_d)$ and $(C,\mathfrak{q})=(A/J,\mathfrak{m}/J)$.\\
	For $d=1$,  in this case $\rho_2 = \ell({\mathfrak{m}}^3/{x_1\mathfrak{m}^2}) =0$, so $G(A)$ is Cohen-Macaulay. So $e_i(A)=e_i(C)$ for all $i\geq 0$. Now from $d=0 $ case, $\type A = \type C\geq e-h-1$.\\
	For $d=2$ in this case from Lemma \ref{sigma}, we have $\sigma_{2}=0$ i.e. $\widetilde{\mathfrak{m}^3}=J\widetilde{\mathfrak{m}^2}$. So $\mathfrak{q}^3=0$. In this case  $\mathfrak{q}^2\neq0$, because if  $\mathfrak{q}^2=0$ then $C$ has minimal multiplicity and this gives $A$ has minimal multiplicity which is not possible as $e_2(A)\neq0$. Now since $\mathfrak{q}^2\subseteq$ Soc$(C)$ and type$(C)=$ type$(A)$, so type$(A)\geq e-h-1$.\\ Now for $d\geq3$, set $(B,\mathfrak{n})=(A/(x_1,\ldots,x_{d-2}),\mathfrak{m}/(x_1,\ldots,x_{d-2}))$ then
	
	 $e_i(A)=e_i(B)$ for $i=0,1,2$ and type$(A)=$ type$(B)$\\ So by  $d=2$ case type$(A)\geq e-h-1$.
	\end{proof}
\section{$\mathbf{type(A)= e-h-1}$}
In this section we prove that if  $e_2=e_1-e+1\neq0$ .
	and if type$(A)=e-h-1$, then $G(A)$ is Cohen-Macaulay.

\s \label{incl} Let $x$ be $A$-superficial. Set $(B,\n) = (A/(x), \m/(x))$.
When $\dim B > 0$  that we have an inclusion (see \cite[2.9]{Pu1})
\[
0 \rt \frac{\widetilde{\m^2}}{\m^2} \rt \frac{\widetilde{\n^2}}{\n^2}. \tag{*}
\]
If $\dim B = 0$ then we have an obvious map
\[
\phi \colon \frac{\widetilde{\m^2}}{\m^2} \rt \frac{\n}{\n^2}. \tag{**}
\]
We show $\phi$ is an inclusion. Indeed, if $\phi(\overline{a}) = 0$ then $a = b + rx$ for some $b \in \m^2$. Thus $a-b \in \widetilde{\m^2}$. So $r \in (\widetilde{\m^2} \colon x) = \m$. It follows that $a \in \m^2$.

\begin{thm}\label{A}
Let $(A,\mathfrak{m},k)$ be a Cohen-Macaulay local ring with dimension $d$ and suppose $e_2=e_1-e+1\neq0$ .
	If type$(A)=e-h-1$, then $G(A)$ is Cohen-Macaulay.
\end{thm}
\begin{proof} We will prove the theorem by induction on dimension $d$.\\
	If dim $A=0$ then  $G$ is Cohen-Macaulay.\\ Let  $x_1,\ldots,x_d$ be an $A$ superficial sequence. Set $J=(x_1,\ldots,x_d)$. \\ Now if $d=1$ then  since $e_2=e_1-e+1$ so $\rho_2 = \ell({\mathfrak{m}}^3/{x_1\mathfrak{m}^2}) =0$ this implies $G$ is Cohen-Macaulay.\\ Now let $d=2$, set $(B,\mathfrak{n})=(A/(x_1),\mathfrak{m}/(x_1))$ and $(C,\mathfrak{q})=(A/(x_1,x_2),\mathfrak{m}/(x_1,x_2))$
	Since $e_2=e_1-e+1$, so from [Lemma \ref{sigma}], we have $\sigma_{2}=0$ i.e. $\widetilde{\mathfrak{m}^3}=J\widetilde{\mathfrak{m}^2}$, so we get $\mathfrak{q}^3=0$.\\ Now from \cite[Prop.4.4]{Pu2} we have $\overline{\widetilde{\mathfrak{m}^3}}=\widetilde{\mathfrak{n}^3}$, so  $\widetilde{\mathfrak{m}^3}={J\widetilde{\mathfrak{m}^2}}$ gives $\widetilde{\mathfrak{n}^3}={\bar{J}\widetilde{\mathfrak{n}^2}}$. So, $\mathfrak{n}.\widetilde{\mathfrak{n}^2}\subseteq \widetilde{\mathfrak{n}^3}\subseteq (x_2)B$.
	
	 This will give us , $\overline{\mathfrak{n}.\widetilde{\mathfrak{n}^2}} = \mathfrak{q}\overline{\widetilde{\mathfrak{n}^2}}=0$ so $\overline{\widetilde{\mathfrak{n}^2}}\subseteq Soc(C)$.\\
	So we get that, $$\mathfrak{q}^2\subseteq \overline{\widetilde{\mathfrak{n}^2}}\subseteq Soc(C).$$   Now we have $\ell(\mathfrak{q}^2)=e-h-1$  and dim$_kSoc(C)=e-h-1$, So $\mathfrak{q}^2 =\overline{\widetilde{\mathfrak{n}^2}}$.

By (**) we get $\widetilde{\n^2} = \n^2$ and so by (*) we get $\widetilde{\m^2} = \m^2$. 	
	 Now since  $\sigma_n=0$ for all $n\geq2$, so $\widetilde{\mathfrak{m}^i}=\mathfrak{m}^i$  for all $i\geq1$. This implies depth $G(A)>0$, So from [Lemma\ref{dep02}],\hspace{2mm}  depth $G(A)=2$ so $G(A)$ is Cohen-Macaulay.

 Now if $d\geq3$, Set $(B,\mathfrak{n})=(A/(x_1,\ldots,x_{d-2}),\mathfrak{m}/(x_1,\ldots,x_{d-2}))$ then $B$ is Cohen-Macaulay ring of dimension 2 with $e_2(B)=e_1(B)-e(B)+1$ and \\ type$(B)$=type$(A)=e-h-1$, so $G(B)$ is Cohen-Macaulay and by Sally-descent $G(A)$ is Cohen-Macaulay.

\end{proof}

\begin{remark}
	If $(A,\mathfrak{m})$ is a Cohen-Macaulay local ring of dimension $d$, $e_2=e_1-e+1 \neq 0$ and type $A= e-h-1$ then $$H_A(z)=\frac{1+hz+(e-h-1)z^2}{(1-z)^d}.$$
\end{remark}

\section{$\mathbf{type(A) = e-h} $ }
In this section we describe our results when $e_2 = e_1 - e + 1 \neq 0$ and  $\type(A) = e - h $. We first discuss the case when $\dim A = 2$.
\begin{thm}\label{d=2}
	Let $(A,\mathfrak{m})$ be a Cohen-Macaulay local ring of dimension two with type$(A)=e-h$ and $x,y$ be $A$-superficial sequence. Set $J=(x,y)$. Suppose $e_2=e_1-e+1\neq0$  then
	\begin{enumerate}[\rm (1)]
		\item depth $G(A)=0$ or $2$.
	
		\item If depth $G(A)=0$ then
	
		\begin{enumerate}[\rm (a)]
		
			\item $\widetilde{\mathfrak{m}^{n+1}}=J\widetilde{\mathfrak{m}^n}$ for all $n\geq2$.
		
			\item $\widetilde{G}(A)$ is Cohen-Macaulay.
		
			\item $\widetilde{\mathfrak{m}^2} \neq \mathfrak{m}^2$ and  $\ell({\widetilde{\mathfrak{m}^2}}/{\mathfrak{m}^2})=1$.
		
			\item $\widetilde{\mathfrak{m}^j}=\mathfrak{m}^j$ for all $j\geq3$ and $\ell({\mathfrak{m}^3/{J\mathfrak{m}^2}})=2$.
		\end{enumerate}
	\end{enumerate}
\end{thm}
\begin{proof}

	Set $(B,\mathfrak{n})=(A/(x),\mathfrak{m}/(x))$ and $(C,\mathfrak{q})=(A/J,\mathfrak{m}/J)$.
	
	\begin{enumerate}
		\item This is Lemma \ref{dep02}.
	
		\item Assume depth $G(A)=0$
	
		\begin{enumerate}
			\item
		 Since $e_2=e_1-e+1$ we get $\sigma_{i}=0$ for all $i\geq2$ from [Lemma\ref{sigma}]. So $\widetilde{\mathfrak{m}^{n+1}}=J\widetilde{\mathfrak{m}^n}$ for all $n\geq2$.

	\item Since  $\widetilde{\mathfrak{m}^{n+1}}\cap J = J\widetilde{\mathfrak{m}^n}$ for all $n\geq1$. So $\widetilde{G}(A)$ is Cohen-Macaulay (see\cite[Prop 3.5]{HM}).

	\item If $\widetilde{\mathfrak{m}^2}=\mathfrak{m}^2$ then from (a), $\widetilde{\mathfrak{m}^i}=\mathfrak{m}^i$ for all $i\geq1$. This gives depth $G(A)>0$, which is contradiction. So $\widetilde{\mathfrak{m}^2}\neq \mathfrak{m}^2$. \\
	 We know that $\mathfrak{q}^2 \subseteq \overline{\widetilde{\mathfrak{n}^2}}\subseteq Soc(C)$. Furthermore $\ell(\mathfrak{q}^2)=e-h-1$ and dim$_k(Soc(C))=e-h$.
Consider the map $\phi \colon \widetilde{\n^2}/\n^2 \rt \q/\q^2$ from \ref{incl}(ii). It follows that the image of $\phi$ lands in $Soc(C)/\q^2$ and this has length $= 1$. So it follows that
$\ell(\widetilde{\n^2}/\n^2) \leq 1$. Again from \ref{incl} it follows that $\ell(\widetilde{\m^2}/\m^2) \leq 1$. As $\widetilde{\m^2} \neq \m^2$ we get that $\ell(\widetilde{\m^2}/\m^2) = 1$.

	 \item  Now since $\ell(\widetilde{\mathfrak{m}^2}/{\mathfrak{m}^2})=1$, so $\widetilde{\mathfrak{m}^2}=\mathfrak{m}^2+(a)$ where $a\notin \mathfrak{m}^2$. From(a), $\widetilde{\mathfrak{m}^3}=J\widetilde{\mathfrak{m}^2}$ so we have $$\widetilde{\mathfrak{m}^3}=J(\mathfrak{m}^2+(a))=J\mathfrak{m}^2+aJ\subseteq \mathfrak{m}^3+aJ$$
	
	this will give us $$\ell(\widetilde{\mathfrak{m}^3}/{\mathfrak{m}^3})\leq2 \quad \text{and} \quad \ell(\widetilde{\mathfrak{m}^3}/{J\mathfrak{m}^2})\leq2$$  Now
	$\ell(\widetilde{\mathfrak{m}^3}/{J\mathfrak{m}^2})\neq0$ otherwise $\mathfrak{m}^3=J\mathfrak{m}^2$ which is not possible as from our assumption depth $G=0$.\\ Now if $\ell(\widetilde{\mathfrak{m}^3}/{J\mathfrak{m}^2})=1$ then $\ell(\mathfrak{m}^3/{J\mathfrak{m}^2})=0$ or $1$.
	Now as $\ell(\mathfrak{m}^3/{J\mathfrak{m}^2})=0$ implies  $G(A)$ is Cohen-Macaulay and $\ell(\mathfrak{m}^3/{J\mathfrak{m}^2})=1$ implies depth of  $G(A)$ is positive from\cite[Cor 2.7]{H}. So $\ell(\widetilde{\mathfrak{m}^3}/{J\mathfrak{m}^2})\neq1$. \\
	Therefore only possibility is $\ell(\widetilde{\mathfrak{m}^3}/{J\mathfrak{m}^2})=2$.
	And from short exact sequence $$0\rightarrow \frac{\mathfrak{m}^3}{J\mathfrak{m}^2}\rightarrow \frac{\widetilde{\mathfrak{m}^3}}{J\mathfrak{m}^2}\rightarrow \frac{\widetilde{\mathfrak{m}^3}}{\mathfrak{m}^3}\rightarrow0.$$  we get $\ell(\mathfrak{m}^3/{J\mathfrak{m}^2})=2$ and $\widetilde{\mathfrak{m}^3}=\mathfrak{m}^3$, so from(a) $\widetilde{\mathfrak{m}^j}=\mathfrak{m}^j$ for all $j\geq3$.

\end{enumerate}	
		
	\end{enumerate}
\end{proof}
Now we will consider general case that is dimension greater than two case.
\begin{thm} \label{B}
Let $(A,\mathfrak{m})$ be Cohen-Macaulay local ring of dimension $d\geq2$ with type$(A)=e-h$ and $e_2=e_1-e+1\neq0$. Suppose $x_1,\ldots,x_d$ is $A$-superficial sequence. Set $J=(x_1,\ldots,x_d)$. Now if depth $G(A)=0$ then
\begin{enumerate}[\rm (1)]

\item $\widetilde{\mathfrak{m}^2} \neq \mathfrak{m}^2$ and  $\ell({\widetilde{\mathfrak{m}^2}}/{\mathfrak{m}^2})=1$.
\item $\widetilde{G}(A)$ is Cohen-Macaulay.
\item  $\widetilde{\mathfrak{m}^{n+1}}=J\widetilde{\mathfrak{m}^n}$ for all $n\geq2$.
\item $\widetilde{\mathfrak{m}^j}=\mathfrak{m}^j$ for all $j\geq3$ and $\ell({\mathfrak{m}^3/{J\mathfrak{m}^2}})=d$.
\end{enumerate}
\end{thm}

\begin{proof}
We will prove this result by induction on dimension of $A$.\\ When $d=2$ then see [Theorem \ref{d=2}]. Now assume $d\geq3$ and result is true for dimension $d-1$.
Set $(B,\mathfrak{n})=(A/(x_1),\mathfrak{m}/(x_1))$ and $\bar{J}=(x_2,\ldots,x_d)B$. We note that by Sally descent we have $\depth G(B) = 0$.
\begin{enumerate}
\item We have an exact sequence from \cite[5.5]{Pu1}
\begin{equation*}
0\rightarrow \frac{(\mathfrak{m}^{i+1}:x_1)}{\mathfrak{m}^i}\rightarrow \frac{\widetilde{\mathfrak{m}^i}}{\mathfrak{m}^i}\rightarrow \frac{\widetilde{\mathfrak{m}^{i+1}}}{\mathfrak{m}^{i+1}}\rightarrow \frac{\widetilde{\mathfrak{n}^{i+1}}}{\mathfrak{n}^{i+1}}
\end{equation*}

Now since $\widetilde{\mathfrak{m}}=\mathfrak{m}$, So we have exact sequence

\begin{equation*}
0\rightarrow \frac{\widetilde{\mathfrak{m}^2}}{\mathfrak{m}^2} \rightarrow \frac{\widetilde{\mathfrak{n}^2}}{\mathfrak{n}^2}
\end{equation*}
and by induction assumption $\ell(\widetilde{\mathfrak{n}^2}/\mathfrak{n}^2)=1$ , So $\ell(\widetilde{\mathfrak{m}^2}/\mathfrak{m}^2)\leq1$.\\ Since $\widetilde{\mathfrak{n}^i}=\mathfrak{n}^i$  for all $i\geq3$. So we have exact sequance for all $i\geq2$

\begin{equation*}
\frac{\widetilde{\mathfrak{m}^i}}{\mathfrak{m}^i}\rightarrow \frac{\widetilde{\mathfrak{m}^{i+1}}}{\mathfrak{m}^{i+1}}\rightarrow0
\end{equation*}
Now if $\widetilde{\mathfrak{m}^2}=\mathfrak{m}^2$ then $\widetilde{\mathfrak{m}^j}=\mathfrak{m}^j$ for all $j\geq1$  which gives depth $G(A)>0$ a contradiction. So $\ell(\widetilde{\mathfrak{m}^2}/\mathfrak{m}^2)=1$.

\item Since $$ 0\rightarrow \frac{\widetilde{\mathfrak{m}^2}}{\mathfrak{m}^2} \rightarrow \frac{\widetilde{\mathfrak{n}^2}}{\mathfrak{n}^2} \quad \text{and} \quad \ell(\widetilde{\mathfrak{m}^2}/\mathfrak{m}^2)=\ell(\widetilde{\mathfrak{n}^2}/\mathfrak{n}^2)=1 $$ we have $\overline{\widetilde{\mathfrak{m}^2}}=\widetilde{\mathfrak{n}^2}$. Also from induction $\widetilde{\mathfrak{n}^i}=\mathfrak{n}^i$  for all $i\geq3$. So we have $$\mathfrak{n}^j \subseteq \overline{\widetilde{\mathfrak{m}^j}}\subseteq \widetilde{\mathfrak{n}^j}=\mathfrak{n}^j$$
So $\overline{\widetilde{\mathfrak{m}^j}}=\widetilde{\mathfrak{n}^j}$ for all $j\geq1$. Thus the Ratliff-Rush filtration on $A$ behaves well with respect to superficial element $x_1$. Now $\widetilde{G}(A)/(x_1^*)=\widetilde{G}(B)$ and $\widetilde{G}(B)$ is Cohen-Macaulay and $x_1^*$ is $\widetilde{G}(A)$-regular. So we get $\widetilde{G}(A)$ is Cohen-Macaulay.

\item Now since $\overline{\widetilde{\mathfrak{m}^j}}=\widetilde{\mathfrak{n}^j}$ for all $j\geq1$. So from [\ref{nn}], the map  $$\eta_i : \widetilde{\mathfrak{m}^{i+1}}/{J\widetilde{\mathfrak{m}^i}} \rightarrow \widetilde{\mathfrak{n}^{i+1}}/{\bar{J}\widetilde{\mathfrak{n}^i}} \quad \text{for all} \quad i\geq0 $$ is isomorphism. Now by induction assumption $\widetilde{\mathfrak{n}^{i+1}}=\bar{J}\widetilde{\mathfrak{n}^i}$ for all $i\geq2$,  so $\widetilde{\mathfrak{m}^{i+1}}=J\widetilde{\mathfrak{m}^i}$ for all $i\geq2$.
\item
We want to prove $\widetilde{\mathfrak{m}^i}=\mathfrak{m}^i$ for all $i\geq3$. Since we have proved that $\widetilde{\mathfrak{m}^{i+1}}=J\widetilde{\mathfrak{m}^i}$ for all $i\geq2$.  So to prove the result it is sufficient to prove $\widetilde{\mathfrak{m}^3}=\mathfrak{m}^3$.  Now if possible assume that $\widetilde{\mathfrak{m}^3}\neq\mathfrak{m}^3$. For  any $A$-superficial element $x$, we have exact sequence $$0\rightarrow \frac{(\mathfrak{m}^3:x)}{\mathfrak{m}^2}\rightarrow \frac{\widetilde{\mathfrak{m}^2}}{\mathfrak{m}^2}\rightarrow \frac{\widetilde{\mathfrak{m}^3}}{\mathfrak{m}^3}\rightarrow0 $$
So, $\ell(\widetilde{\mathfrak{m}^3}/{\mathfrak{m}^3})=1$ and $(\mathfrak{m}^3:x)=\mathfrak{m}^2$.\\
As $G(A)$ has depth zero and it is generalized Cohen-Macaulay by \ref{GCM}. It follows that \begin{equation*}
Ass (G(A)) = \{\M, P_1,\ldots,P_r\}
\end{equation*}
where $P_1,\ldots,P_r$ are minimal primes of $G(A)$ and $\M$ is the maximal homogeneous ideal of $G(A)$. Now by \ref{GCM}, dim $G(A)/P_i=d\geq3$ for all $i$. Let $V$ be $k$-vector space $\mathfrak{m}/\mathfrak{m}^2$. If $V\subset P_i$ then  $P_i=\M$, a contradiction.

 In particular, dim $_k(V\cap P_i)\leq $ dim $_k V-1$ for all $i.$\\
 {\bf Claim :}  dim $_k(V\cap P_i)\leq $ dim $_k V-2$ for all $i.$\\
 Now if possible assume that for some $i$, dim $_k(V\cap P_i)= $ dim $_k V-1$. Then there exists $s\in V$ such that $P_i\cap V +ks = V$. So we get $P_i+G(A)s=\M$. This gives dim $G(A)/P_i\leq1$ which is a contradiction. So we have proved the claim.

 As $k$ is infinite there exists $k$-linearly independent vectors $u^*,v^*$ in $V$ such that if $H=ku^*+kv^*$ then $H\cap P_i=0$ for all $i$. It follows that $u-\alpha v$ is an $A-$superficial element for any unit $\alpha$ (see\ref{superf}). Furthermore $u,v$ are $A-$ superficial.

Now since $\ell(\widetilde{\mathfrak{m}^2}/{\mathfrak{m}^2})=1$, we can write $ \widetilde{\mathfrak{m}^2}=\mathfrak{m}^2+(a)$ where $ a\notin \mathfrak{m}^2$. Let $x$ be any $A$-superficial element. If $xa\in \mathfrak{m}^3$ then $a\in (\mathfrak{m}^3:x)=\mathfrak{m}^2$ a contradiction. Thus $\overline{xa}$ is a non-zero element in  $\widetilde{\mathfrak{m}^3}/\mathfrak{m}^3$.
In particular $\overline{ua},\overline{va}$ are non zero elements of $\widetilde{\mathfrak{m}^3}/\mathfrak{m}^3$. As $\widetilde{\mathfrak{m}^3}/\mathfrak{m}^3 \cong k$ we get that there is a unit $\alpha\in A$ such that  $\overline{ua}=\alpha\overline{va}$. So $(ua-\alpha va)\in \mathfrak{m}^3$. By construction $u-\alpha v$ is an $A$-superficial element. So $a\in \mathfrak{m}^3:(u-\alpha v)=\mathfrak{m}^2$ a contradiction. \\ So $\widetilde{\mathfrak{m}^3}=\mathfrak{m}^3$ and $\ell({(\mathfrak{m}^3:x)}/{\mathfrak{m}^2})=1$.\\ Now from Lemma \ref{new} we have  short exact sequence $$0\rightarrow \frac{(\mathfrak{m}^3:x)}{\mathfrak{m}^2}\rightarrow \frac{\mathfrak{m}^3}{J\mathfrak{m}^2}\rightarrow \frac{\mathfrak{n}^3}{\bar{J}\mathfrak{n}^2}\rightarrow0 $$ And from induction assumption $\ell({\mathfrak{n}^3}/{\bar{J}\mathfrak{n}^2})=d-1$, so $\ell(\mathfrak{m}^3/{J\mathfrak{m}^2})=d$.
\end{enumerate}
\end{proof}

\subsection{Hilbert Series}

Let $(A,\mathfrak{m})$ be Cohen-Macaulay local ring with maximal ideal $\mathfrak{m}$ and dimension $d\geq2$. Suppose $x_1,\ldots,x_d$ is $A$-superficial sequence.\\ Set $J=(x_1,\ldots,x_d)$ and $(B,\mathfrak{n})=(A/{(x_1,\ldots,x_{d-2})},\mathfrak{m}/(x_1,\ldots,x_{d-2}))$.

 The associated graded ring with respect to Ratliff-Rush filtration is $$\widetilde{G}(A)=\bigoplus_{n\geq0}{\frac{\widetilde{\mathfrak{m}^n}}{\widetilde{\mathfrak{m}^{n+1}}}}$$
Its Hilbert series $$\sum_{n\geq0}\ell({\widetilde{\mathfrak{m}^n}}/{\widetilde{\mathfrak{m}^{n+1}}})z^n=\frac{\tilde{h}_A(z)}{(1-z)^d}$$

Now we can write $h_A(z)=\tilde{h}_A(z)+(1-z)^{d+1}r(z)$ where $h_A(z)$ is the $h$ polynomial of $A$ and $r(z)=\sum_{n\geq0}\ell(\widetilde{\mathfrak{m}^{n+1}}/{\mathfrak{m}^{n+1}})z^n\in \mathbb{Z}[z]$.\\ If the dimension is two then from\cite[Theorem 3]{It} $$\tilde{h}_{A}(z)=1+(\sigma_0-\sigma_1)z+\ldots+(\sigma_{s-1}-\sigma{s})z^{s}.$$ where $\sigma_n=\ell(\widetilde{\mathfrak{m}^{n+1}}/{J\widetilde{\mathfrak{m}^n}})$.

\begin{thm}
If in addition with above setup, the ring $A$ satisfies $e_2=e_1-e+1\neq0$ , type$(A)=e-h$ and depth $G(A)=0$ then $h_A(z)=1+(h-1)z+(e-h)z^2+z(1-z)^{d+1}$.\\
In general, if depth $G(A)=n\neq d$ then $n\leq d-2$ and $h_A(z)=1+(h-1)z+(e-h)z^2+z(1-z)^{d-n+1}$, and if $G(A)$ is Cohen-Macaulay then $h_A(z)=1+hz+(e-h-1)z^2$.
\end{thm}

\begin{proof}
We are using above notation, to prove the theorem we will first assume depth $G(A)=0$.\\
Now if $d=2$ then since $e_2=e_1-e+1$ we have $\sigma_n=0$ for all $n\geq2$.\\
So $$\tilde{h}_A(z)=1+(\sigma_0-\sigma_1)z+\sigma_1z^2$$ Now as depth $G(A)=0$ so $$\sigma_0=\ell(\mathfrak{m}/J)=e-1 \quad \text{and} \quad \sigma_1=\ell(\widetilde{\mathfrak{m}^2}/{J\mathfrak{m}})=e-h=type(A)$$ So $$\tilde{h}_A(z)=1+(h-1)z+(e-h)z^2$$

Now if $d\geq3$ then since $\tilde{G}(A)$ is Cohen-Macaulay,

In this case Ratliff-Rush filtration behaves well with respect to superficial element and $\widetilde{G}(A)/(x_1^*,\ldots,x_{d-2}^*)=\widetilde{G}(B)$ see Theorem \ref{B}, so $\tilde{h}_A(z)=\tilde{h}_B(z)$.\\ Also since $$\ell(\widetilde{\mathfrak{m}^2}/\mathfrak{m}^2)=1 \quad \text{and} \quad \widetilde{\mathfrak{m}^i}=\mathfrak{m}^i \quad \text{for all} \quad i\geq3$$

this gives $r(z)=z$. So, $h$ polynomial of $A$ is $$h_A(z)=1+(h-1)z+(e-h)z^2+z(1-z)^{d+1}$$

Now suppose depth $G(A)=n>0$.

We can assume that $d\geq3$ and $d-n\geq2$, (because if $d-n=0$ then $G(A)$ is Cohen-Macaulay). Also note $d-n=1$ is not possible because in dimension two  $\depth G(A)$ is zero or two see [\ref{dep02}].)

Suppose $\alpha_1,\ldots,\alpha_n$ is a super-regular sequence in $A$, then $R=A/(\alpha_1,\ldots,\alpha_n)$ is Cohen-Macaulay local ring of dimension $d-n$ and depth of its associated graded ring with respet to the maximal ideal is zero.

So, $h_R(A)=1+(h-1)z+(e-h)z^2+z(1-z)^{d-n+1}$

Now since $\alpha_1,\ldots,\alpha_n$ is a super-regular sequence, so $$h_A(z)=h_R(z)=1+(h-1)z+(e-h)z^2+z(1-z)^{d-n+1}$$
\end{proof}

So from the above theorem we can calculate Hilbert series $H_A(z)$, since $$H_A(z)=\frac{h_A(z)}{(1-z)^d}.$$

Now  we will conclude the section by  proving a result which gives a  bound for $d$, above which depth of $G(A)$ cannot be zero.

\begin{thm}
	Let $(A,\mathfrak{m})$ be Cohen-Macaulay local ring with dimension $d$. Suppose $e_2=e_1-e+1\neq0$, type$(A)=e-h=k$. Now if $d>(k-1)^{<2>}$ then depth of $G(A)$ is non-zero.
\end{thm}
	\begin{proof}
	Let $x_1,\ldots,x_d$ be an $A$-superficial sequence. Set $J=(x_1,\ldots,x_d)$.

	Assume depth $G(A)=0$, Now we will show that $d\leq (k-1)^{<2>}$

	 Now consider $G'= G(A)/(x_1^*,\ldots,x_d^*)$. Clearly, $G'$ is finite dimensional $k$-algebra.
	
	  Since $\widetilde{\mathfrak{m}^i}=\mathfrak{m}^i$ for all $i\geq3$ and $\widetilde{\mathfrak{m}^{i+1}}=J\widetilde{\mathfrak{m}^i}$ for all $i\geq2$.
	
	  So $$G'=A/\mathfrak{m}\oplus\mathfrak{m}/(J+\mathfrak{m}^2)\oplus\mathfrak{m}^2/{J\mathfrak{m}}\oplus \mathfrak{m}^3/{J\mathfrak{m}^2}$$
	   Now Hilbert function $H(G',2)=\ell(\mathfrak{m}^2/{J\mathfrak{m}})=e-h-1=k-1$
	
	   and $H(G',3)=\ell(\mathfrak{m}^3/{J\mathfrak{m}^2})=d$ from[Theorem \ref{B}]
	
	   Now from Macaulay theorem\cite[Theorem4.2.10]{BH}
	    $$H(G',3)\leq H(G',2)^{<2>}$$
	    this gives $$d\leq(k-1)^{<2>}.$$
	\end{proof}

\section{Examples}
In this section we discuss a few examples which illustrates our results.
\begin{example}
	(Wang's example) see \cite[3.10]{CPR}
	$$A = k[[ x, y, z, u, v]]/(z^2, zu, zv, uv, yz - u^3,  xz-v^3).$$ $A$ is \CM \ local ring of dimension 2 and depth $G(A)=0$. We have $h_A(z)=1+3z+3z^3-z^4$, $e_2=e_1-e+1=3$ and type$(A)=e-h=3.$
\end{example}
\begin{example}(Extension of Wang's example )(For these calculations we have used COCOA \cite{CO})
	\begin{enumerate}
		\item  $A=k[[x,y,t,z,u,v,w]]/(z^2,zu,zv,uv,zw,uw,vw,yz-u^3,xz-v^3,tz-w^3)$
		$A$ is Cohen-Macaulay local ring of dimension 3 and depth $G(A)=0.$ We have $h_A(z)=1+4z+6z^3-4z^4+z^5$, $e_2=e_1-e+1=4$ and type$(A)=e-h=4$
		\item  $A=k[[x,y,t,p,z,u,v,w,q]]/(z^2,zu,zv,zw,zq,uv,uw,uq,vw,vq,wq,xz-u^3,yz-v^3,tz-w^3,pz-q^3)$
		
		$A$ is \CM \ local ring of dimension 4 and depth $G(A)=0$. We have $h_A(z)=1+5z+10z^3-10z^4+5z^5-z^6$, $e_2=e_1-e+1=5$ and type$(A)=e-h=5.$
	\end{enumerate}
\end{example}
\begin{example}
	Let $(A,\mathfrak{m})$ be Gorenstein local ring with $e=h+2$, then clearly $e_2=e_1-e+1$ and type$(A)=e-h-1=1$, see \cite{S}.
\end{example}


\begin{thebibliography}{10}
	
	\bibitem{CO}
	J. Abbott,  A. M. Bigatti  and L. Robbiano,
\emph{{CoCoA}: a system for doing {C}omputations in {C}ommutative {A}lgebra},
Available at \texttt{http://cocoa.dima.unige.it}.
	
	\bibitem{BH}
	W.~Bruns and J.~Herzog, \emph{{Cohen-Macaulay rings}}, vol.~39, Cambridge
	studies in advanced mathematics, Cambridge University Press,~Cambridge, 1993.
	
	\bibitem{CPR}
	A.~Corso, C.~Polini and M.~E.~Rossi,
	\emph{Depth of associated graded rings via Hilbert coefficients of ideals},
	J. Pure Appl. Algebra 201 (2005), no. 1-3, 126--141.
	
	\bibitem{H}
	S.~Huckaba,
	\emph{On associated graded rings having almost maximal depth},
	Comm. Algebra 26 (1998), no. 3, 967-–976.
	
	\bibitem{HM}
	S.~Huckaba and T.~Marley,
	\emph{Hilbert coefficients and the depths of associated graded rings},
	J. London Math. Soc. 56 (2) (1997) 64–-76.
	
	\bibitem{It}
	S.~Itoh, \emph{{Hilbert coefficients of integrally closed ideals }}, J.~Algebra
	\textbf{176} (1995), 638--652.

\bibitem{NR}
 D.~G.~Northcott and D.~Rees,
 \emph{Reductions of ideals in local rings},
  Proc. Cambridge Philos. Soc. 50 (1954), 145–-158.

    \bibitem{Pu0}
	T.~J.~Puthenpurakal,
     \emph{Hilbert-coefficients of a Cohen-Macaulay module},
     J. Algebra 264 (2003), no. 1, 82–-97.

	\bibitem{Pu1}
	\bysame,
	\emph{Ratliff-{R}ush filtration, regularity and depth of higher associated graded modules. {I}},
	J. Pure Appl. Algebra \textbf{208} (2007), no.~1, 159--176.
	
		\bibitem{Pu2}
	\bysame,
	\emph{Ratliff-{R}ush filtration, regularity and depth of higher associated graded modules. {II}},
	J. Pure Appl. Algebra \textbf{221} (2017), no.~3,
	611--631.
	
		\bibitem{rv}
	M.~E.~Rossi and G.~Valla,
	\emph{A conjecture of J. Sally},
	Comm. Algebra 24 (1996), no. 13, 4249–-4261.

	
	\bibitem{S}
	J.~D.~Sally,
	\emph{Tangent cones at Gorenstein singularities},
	Compositio Math. 40 (1980) 167–-175.
	
	
	
	\bibitem{V}
	G.~Valla,
	\emph{Problems and results on Hilbert functions of graded algebras},
	in: Six Lectures in Commutative Algebra, in: Progr. Math., Vol. 166, Birkhäuser, Basel, 1998, pp. 293–-344.
\end{thebibliography}
\end{document}